\documentclass[12pt,leqno]{article}
\usepackage{amsmath, amsthm, amsfonts, amssymb, color}
\usepackage{mathrsfs}
\newtheorem{thm}{Theorem}[section]

\newtheorem{lem}[thm]{Lemma}
\newtheorem{prp}[thm]{Proposition}
\newtheorem{rem}[thm]{Remark}

\theoremstyle{definition}

\newcommand{\scr}[1]{\mathscr #1}
\definecolor{wco}{rgb}{0.5,0.2,0.3}

\numberwithin{equation}{section} \theoremstyle{remark}

\newcommand{\ua}{\uparrow}

\title{{\bf Poincar\'{e} and Weak Poincar\'{e} Inequalities for the Mixed Poisson Measures}}
\author{{\bf Chang-Song Deng}\footnote{Email: dengcs@mail.bnu.edu.cn (C.-S. Deng)}\\
\footnotesize{School of Mathematics and Statistics, Wuhan University, Wuhan 430072, China}}
\date{}

\bigskip

\begin{document}
\def\R{\mathbb R}  \def\ff{\frac} \def\ss{\sqrt} \def\B{\mathbf
B}
\def\N{\mathbb N} \def\kk{\kappa} \def\m{{\bf m}}
\def\dd{\delta} \def\DD{\Delta} \def\vv{\varepsilon} \def\rr{\rho}
\def\<{\langle} \def\>{\rangle} \def\GG{\Gamma} \def\gg{\gamma}
  \def\nn{\nabla} \def\pp{\partial} \def\EE{\scr E} \def\BB{\scr B}
\def\d{\text{\rm{d}}} \def\bb{\beta} \def\aa{\alpha} \def\D{\scr D}
  \def\si{\sigma} \def\ess{\text{\rm{ess}}}
\def\beg{\begin} \def\beq{\begin{equation}}  \def\F{\scr F}
\def\Ric{\text{\rm{Ric}}} \def\Hess{\text{\rm{Hess}}}
\def\e{\text{\rm{e}}} \def\ua{\underline a} \def\OO{\Omega}  \def\oo{\omega}
 \def\tt{\tilde} \def\Ric{\text{\rm{Ric}}}
\def\cut{\text{\rm{cut}}} \def\P{\mathbb P} \def\ifn{I_n(f^{\bigotimes n})}
\def\C{\scr C}      \def\aaa{\mathbf{r}}     \def\r{r}
\def\gap{\text{\rm{gap}}} \def\prr{\pi_{{\bf m},\varrho}}  \def\r{\mathbf r}
\def\Z{\mathbb Z} \def\vrr{\varrho}
\def\L{\scr L}
\def\Tt{\tt} \def\TT{\tt}\def\II{\mathbb I}
\def\i{{\rm in}}\def\Sect{{\rm Sect}}\def\E{\mathbb E} \def\H{\mathbb H}
\def\M{\scr M}\def\Q{\mathbb Q} \def\texto{\text{o}} \def\LL{\Lambda}
\def\TEE{\widetilde{\EE}} \def\B{\mathbf B} \def\i{{\rm i}} \def\supp{{\rm supp}}

\maketitle
\begin{abstract}
By using the Mecke identity, we study a class of birth-death type Dirichlet forms associated with the mixed
Poisson measure. Both Poincar\'{e} and weak Poincar\'{e} inequalities are established, while another Poincar\'{e} type inequality is disproved under some reasonable assumptions.
\end{abstract} \noindent

\noindent

 \vskip 2cm

\section{Introduction}

Let $X$ be a locally compact Polish space with Borel $\si$-field $\scr B(X)$ and let $\mu$ be a Radon measure on $(X,\scr B(X))$.
Let $\GG_{X}$ denote the space of all $\Z_+\cup\{\infty\}$-valued Radon measures on $X$. The set of all $\gg\in\GG_X$ such that $\gg(\{x\})\in\{0,1\}$ for all $x\in X$ is called the configuration space
over $X$. For simplicity, we also call $\GG_X$ the configuration space over $X$.
Endow $\GG_X$ with the vague topology, i.e. the weakest topology on $\GG_X$ such that the mapping
$$
\GG_X\ni\gg\mapsto\<\gg,f\>:=\gg(f)=\int_{X}f\,\d\gg \in\R
$$
is continuous for every $f\in C_0(X)$.
Here $C_0(X)$ denotes the space of all continuous functions on $X$ having compact support.
Denote by $\scr B(\GG_X)$ the corresponding Borel $\sigma$-field on $\GG_X$.
The (pure) Poisson measure with intensity $\mu$, denoted by $\pi_\mu$,  is the unique probability measure on $(\GG_X,\scr B(\GG_X))$
with Laplace transform given by
\begin{equation}\label{LT}
\pi_\mu\left(\e^{\<\cdot,f\>}\right)=\exp\left[\int_X(\e^f-1)\,\d\mu\right],\quad f\in L^1(\mu)\cap L^\infty(\mu).
\end{equation}
Another characteristic of the Poisson measure $\pi_\mu$ is that
for any disjoint sets $A_1,\cdots,A_n\in\scr B(X)$ with $\mu(A_i)<\infty$, $1\leq i\leq n$,
$$\pi_\mu\big(\{\gg\in \GG_X; \ \gg(A_i)= k_i, 1\le i\le n\}\big)=
\prod_{i=1}^n \e^{-\mu(A_i)} \ff{\mu(A_i)^{k_i}}{k_i!},\ \ \
k_i\in \Z_+,\ 1\le i\le n.$$

We remark that for $\mu=0$, $\pi_\mu$ is just the Dirac measure on $(\GG_X,\scr B(\GG_X))$ with total mass in
the empty configuration $\gg=0$ (i.e. the zero measure on $X$). We refer to e.g. \cite{AKR98a,Shi94} for a detailed discussion of the construction of the Poisson measure on configuration space.

Recall the birth-death type Dirichlet form associated with the Poisson measure:
\begin{equation*}
\begin{split}
&\EE_{bd}(F,G)=\int_{\GG_X\times X}\big(F(\gg+\delta_x)-F(\gg)\big)\big(G(\gg+\delta_x)-G(\gg)\big)\,\pi_\mu(\d\gg)\mu(\d x),\\
&\D(\EE_{bd})= \{F\in L^2(\pi_\mu);\,\EE_{bd}(F,F)<\infty\}.
\end{split}
\end{equation*}
It is well known that $(\EE_{bd},\D(\EE_{bd}))$  does not satisfy the log-Sobolev inequality (cf. \cite{Sur84}). In the paper \cite{Wu00}, Wu established the
Poincar\'{e} and $L^1$ log-Sobolev inequalities for $(\EE_{bd},\D(\EE_{bd}))$ by
exploring the martingale representation (the counterpart on Poisson space of the Clark-Oc\^{o}ne formula over Wiener space).
Recently, \cite{WY10} presented a new proof of the Poincar\'{e} inequality for $(\EE_{bd},\D(\EE_{bd}))$. Following the line of \cite{WY10},
the authors of this paper confirmed the $L^1$ log-Sobolev inequality for $(\EE_{bd},\D(\EE_{bd}))$ again in \cite{DS11}.

The central aim of the present paper is to extend the known results concerning Poincar\'{e} type inequalities for the birth-death
type Dirichlet forms associated with the Poisson measure to the mixed Poisson measure case.

The mixed Poisson measure is defined by
$$
\pi_{\lambda,\mu}=\int_{\R^+}\pi_{s\mu}\,\lambda(\d s),
$$
where $\lambda$ is a probability measure on $\R^+:=[0,\infty)$. See \cite{AKR98a, Roe98} for more details on the mixed Poisson measure.
In particular, if $\lambda=\delta_1$, then $\pi_{\lambda,\mu}$ reduces to the (pure) Poisson measure $\pi_\mu$. Another interesting
case is the fractional Poisson measure (cf. \cite{OOSM10}).

Now consider the quadric form
\begin{equation*}
\begin{split}
&\EE(F,G):=\int_{\GG_X\times X\times\R^+}\big(F(\gg+\delta_x)-F(\gg)\big)\big(G(\gg+\delta_x)-G(\gg)\big)\,\pi_{s\mu}(\d\gg)s\mu(\d x)\lambda(\d s),\\
&\D(\EE):=\{F\in L^2(\pi_{\lambda,\mu});\,\EE(F,F)<\infty\}.
\end{split}
\end{equation*}
According to Proposition \ref{dform} below, $(\EE,\D(\EE))$ is a conservative symmetric Dirichlet form on $L^2(\pi_{\lambda,\mu})$ provided
$\int_{\R^+}s\,\lambda(\d s)<\infty$. Now we are ready to present our main results.

\begin{thm}\label{PI}
If $\int_{\R^+}s\,\lambda(\d s)<\infty$, $\mu(X)<\infty$ and there exists a constant $C\geq0$ such that
\begin{equation}\label{con1}
\lambda(f^2)-\lambda(f)^2\leq C\int_{\R^+}s|f'(s)|^2\,\lambda(\d s),\quad f\in C^1_b(\R^+),
\end{equation}
then
$$
\pi_{\lambda,\mu}(F^2)-\pi_{\lambda,\mu}(F)^2\leq\big(1+C\mu(X)\big)\EE(F,F),\quad F\in\D(\EE).
$$
\end{thm}

\begin{rem}
When $\lambda=\delta_1$, obviously (\ref{con1}) holds with $C=0$, and so without the assumption that
$\mu(X)<\infty$ Theorem \ref{PI} reduces to the Poincar\'{e}
inequality for the Poisson measure (see \cite[Remark 1.4]{Wu00}) with the sharp constant $1$ (see \cite{DW10}).
\end{rem}

Next, we consider the weak Poincar\'{e} inequality, which was introduced in \cite{RW01} to describe the non-exponential
convergence rates of Markov semigroups. Let
$$
\|f\|_u=\sup_{s\in\R^+}|f(s)|,\quad\text{and}\quad \|F\|_u=\sup_{\gg\in\GG_X}|F(\gg)|,
$$
where $f$ and $F$ are functions on $\R^+$ and $\GG_X$, respectively.

\begin{thm}\label{WPI}
If $\int_{\R^+}s\,\lambda(\d s)<\infty$, $\mu(X)<\infty$ and there exists $\alpha:(0,\infty)\rightarrow(0,\infty)$ such that
\begin{equation}\label{con2}
\lambda(f^2)-\lambda(f)^2\leq \alpha(r)\int_{\R^+}s|f'(s)|^2\,\lambda(\d s)+r\|f\|_{u}^2,\quad r>0,f\in C^1_b(\R^+),
\end{equation}
then
$$
\pi_{\lambda,\mu}(F^2)-\pi_{\lambda,\mu}(F)^2\leq\big(1+\mu(X)\alpha(r)\big)\EE(F,F)+r\|F\|_{u}^2,\quad r>0,F\in\D(\EE).
$$
\end{thm}

Finally, we point out that under some assumptions, the following Poincar\'{e} type inequality fails:
\begin{equation}\label{PTI}
\pi_{\lambda,\mu}(F^2)\leq C_1\EE(F,F)+C_2\pi_{\lambda,\mu}(|F|)^2,\quad F\in\D(\EE),
\end{equation}
where $C_1$ and $C_2$ are constants.

\begin{prp}\label{PROP}
Assume that $\int_{\R^+}s^2\,\lambda(\d s)\in(0,\infty)$, and there exists a sequence $\{A_n\}_{n\geq1}\subset\scr B(X)$ such that
$\mu(A_n)>0$ for every $n\geq1$ but $\mu(A_n)\rightarrow0$ as $n\rightarrow\infty$ (it is the case if $\mu$ does not have atom). If $C_1<1$, then (\ref{PTI})
fails for any $C_2$.
\end{prp}

The remainder of the paper is organized as follows. In Section 2, we discuss the birth-death type Dirichlet form $(\EE,\D(\EE))$. The proofs of the
main results are addressed in Section 3.

\section{Birth-death type Dirichlet form}

In this section, we shall follow the line of \cite{DW10} to characterize the form $(\EE, \D(\EE))$. First of all, recall
the Mecke identity \cite{Mec67} (see
also \cite{Roe98}):
\begin{equation}\label{mecke}
\int_{\GG_X\times X} H(\gg+\delta_x,
x)\,\pi_\mu(\d\gg)\mu(\d x)= \int_{\GG_X\times X} H(\gg,x)\,\gg(\d x)\pi_\mu(\d\gg)
\end{equation}
holds for any measurable function $H$ on
$\GG_X\times X$ such that one of the above integrals exists.

For $A\in\scr B(\GG_X)$, let
$$
\tilde{A}=\{(\gg,x)\in\GG_X\times X;\ \gg+\delta_x\in A\}.
$$
We first prove the following quasi-invariant property of the mapping
$\gg\mapsto\gg+\delta_x$.

\begin{lem}\label{quasi}
If $\pi_{\lambda,\mu}(A)=0$, then
$$
\int_{\tilde{A}\times\R^+}\pi_{s\mu}(\d\gg)s\mu(\d x)\lambda(\d s)=0.
$$
\end{lem}
\begin{proof}
By the Mecke identity (\ref{mecke}) for $H(\gg,\cdot)=1_A(\gg)$ and using the definition of $\pi_{\lambda,\mu}$, we have
\begin{eqnarray*}
\begin{aligned}
\int_{\tilde{A}\times\R^+}\pi_{s\mu}(\d\gg)s\mu(\d x)\lambda(\d s)
&=\int_{\R^+}\lambda(\d s)\int_{\GG_X\times X}1_A(\gg+\delta_x)\,\pi_{s\mu}(\d\gg)s\mu(\d x)\\
&=\int_{\R^+}\lambda(\d s)\int_{\GG_X\times X}1_A(\gg)\,\gg(\d x)\pi_{s\mu}(\d\gg)\\
&=\int_{\R^+}\lambda(\d s)\int_{A}\gg(X)\,\pi_{s\mu}(\d\gg)\\
&=\int_{A}\gg(X)\,\pi_{\lambda,\mu}(\d\gg)=0.
\end{aligned}
\end{eqnarray*}
\end{proof}

Define the family of cylindrical functions
$$
\F_\mu^C=\big\{\gg\mapsto f(\gg(h_1),\cdots,\gg(h_m));\ m\geq1,f\in C_b^1(\R^m),h_i\in L^1(\mu)\cap L^\infty(\mu)\big\}.
$$

\begin{prp}\label{dform}
Assume that $\int_{\R^+}s\,\lambda(\d s)<\infty$. Then $(\EE,\D(\EE))$ is a conservative symmetric Dirichlet form on $L^2(\pi_{\lambda,\mu})$ with $\D(\EE)\supset\F_\mu^C$.
\end{prp}

\begin{proof}
Due to Lemma \ref{quasi}, $(\EE,\D(\EE))$ is well-defined on $L^2(\pi_{\lambda,\mu})$; that is, the value of $\EE(F,G)$
does not depend on $\pi_{\lambda,\mu}$-versions of $F$ and $G$. Note that $\F_\mu^C$ is dense in $L^2(\pi_{\lambda,\mu})$
and by the definition of $\EE$, the normal contractivity property is trivial.
Therefore, we only need to prove $\F_\mu^C\subset\D(\EE)$ and the closed property of $(\EE,\D(\EE))$.
Let's prove these two points separately.

\medskip

(1) For any $h\in L^1(\mu)\cap L^\infty(\mu)$, noting that
\begin{eqnarray*}
\begin{aligned}
\int_{\GG_X}|\gg(h)|\,\pi_{\lambda,\mu}(\d\gg)&\leq\int_{\R^+}\lambda(\d s)\int_{\GG_X}\gg(|h|)\,\pi_{s\mu}(\d\gg)\\
&=\mu(|h|)\int_{\R^+}s\,\lambda(\d s)<\infty,
\end{aligned}
\end{eqnarray*}
one has $\gg(h)\in\R$ for $\pi_{\lambda,\mu}$-a.e. $\gg\in\GG_X$. Therefore,
$$
F(\gg):=f(\gg(h_1),\cdots,\gg(h_m))\in\F_\mu^C
$$
is well-defined in $L^2(\pi_{\lambda,\mu})$. By the Mean Value Theorem and noting that $h_i\in L^1(\mu)\cap L^\infty(\mu)$, we arrive at
\begin{eqnarray*}
\begin{aligned}
\EE(F,F)&=\int_{\GG_X\times X\times\R^+}\big[f(\gg(h_1)+h_1(x),\cdots,\gg(h_m)+h_m(x))\\
&\qquad\qquad\qquad-f(\gg(h_1),\cdots,\gg(h_m))\big]\,
\pi_{s\mu}(\d\gg)s\mu(\d x)\lambda(\d s)\\
&\leq\|\nabla f\|_\infty^2\int_{\GG_X\times X\times\R^+}\sum_{i=1}^mh_i(x)^2\,\pi_{s\mu}(\d\gg)s\mu(\d x)\lambda(\d s)\\
&=\|\nabla f\|_\infty^2\left(\sum_{i=1}^m\mu(h_i^2)\right)\int_{\R^+}s\,\lambda(\d s)\\
&\leq\|\nabla f\|_\infty^2\left(\sum_{i=1}^m\|h_i\|_\infty\mu(|h_i|)\right)\int_{\R^+}s\,\lambda(\d s)\\
&<\infty.
\end{aligned}
\end{eqnarray*}
Thus, $\F_\mu^C\subset\D(\EE)$.

\medskip

(2) Let $\{F_n\}_{n\geq1}$ be an $\EE_1^{1/2}$-Cauchy sequence, where
$$
\EE_1^{1/2}(F,F):=\sqrt{\EE(F,F)+\pi_{\lambda,\mu}(|F|^2)},\quad F\in\D(\EE).
$$
Since $\{F_n\}_{n\geq1}$ is a Cauchy sequence in $L^2(\pi_{\lambda,\mu})$, there exists $F\in L^2(\pi_{\lambda,\mu})$
such that $F_n\rightarrow F$ in $L^2(\pi_{\lambda,\mu})$. Consequently, we can choose a subsequence $\{F_{n_k}\}_{k\geq1}$
such that $F_{n_k}\rightarrow F$ $\pi_{\lambda,\mu}$-a.e. Then it follows from Lemma \ref{quasi} that
$F_{n_k}(\gg+\delta_x)\rightarrow F(\gg+\delta_x)$
for $\left(\int_{\R^+}\pi_{s\mu}s\lambda(\d s)\right)\times\mu$-a.e. $(\gg,x)\in\GG_X\times X$. So, we obtain from the Fatou
Lemma that
\begin{eqnarray*}
\begin{aligned}
&\EE(F_n-F,F_n-F)\\
&=\int_{\Gamma_X\times X\times\R^+}\liminf_{k\rightarrow\infty}\big[(F_n-F_{n_k})(\gg+\delta_x)-(F_n-F_{n_k})(\gg)\big]^2
\,\pi_{s\mu}(\d\gg)s\mu(\d x)\lambda(\d s)\\
&\leq\liminf_{k\rightarrow\infty}\EE(F_n-F_{n_k},F_n-F_{n_k}).
\end{aligned}
\end{eqnarray*}
Since $\{F_n\}_{n\geq1}$ is an $\EE_1^{1/2}$-Cauchy sequence, this implies that
\begin{equation}\label{tt}
\lim_{n\rightarrow\infty}\EE_1^{1/2}(F_n-F,F_n-F)=0.
\end{equation}
On the other hand, using the Fatou Lemma again, we get
\begin{eqnarray*}
\begin{aligned}
\EE(F,F)&=\int_{\Gamma_X\times X\times\R^+}\liminf_{k\rightarrow\infty}(F_{n_k}(\gg+\delta_x)-F_{n_k}(\gg))^2
\,\pi_{s\mu}(\d\gg)s\mu(\d x)\lambda(\d s)\\
&\leq\liminf_{k\rightarrow\infty}\EE(F_{n_k},F_{n_k})\\
&<\infty
\end{aligned}
\end{eqnarray*}
since $\{F_{n_k}\}_{k\geq1}$ is an $\EE^{1/2}$-Cauchy sequence. Combining this with (\ref{tt}), we conclude
that $F\in\D(\EE)$ and $F_n\rightarrow F$ in $\D(\EE)$ as $n\rightarrow\infty$. Hence, $(\EE,\D(\EE))$ is
closed and the proof is now completed.
\end{proof}

Now we move on to study the regularity of the
Dirichlet form. Note that $\GG_X$ is a Polish space
under the vague topology (see \cite[Proposition 3.17]{Res87}). Since the probability measure $\pi_{\lambda,\mu}$ on the Polish space $\GG_X$ is tight, we can
choose an increasing sequence $\{K_n\}_{n\ge 1}$ consisting of compact subsets of $\GG_X$ such that $\pi_{\lambda,\mu}(K_n^c)\leq1/n$
for any $n\geq1$.
Then $\pi_{\lambda,\mu}$ has full measure on $\GG_X^\mu:=\bigcup_{n=1}^\infty K_n$, which is a locally compact separable metric space.

\beg{prp}   If $\mu(X)<\infty$ and $\int_{\R^+}s\lambda(\d s)<\infty$,
then  $(\EE,\D(\EE))$ is a regular
Dirichlet form on $L^2(\GG_X^\mu; \pi_{\lambda,\mu})$. \end{prp}
\begin{proof}
Since $\mu(X)<\infty$, it is easy to see that $\BB_b(\GG_X^\mu)\subset\D(\EE)$, where $\BB_b(\GG_X^\mu)$ is the set of all
bounded measurable functions on $\GG_X^\mu$. In particular,
$C_0(\GG_X^\mu)\subset\D(\EE)$. Thus, it suffices to prove that
$C_0(\GG_X^\mu)$ is dense in $\D(\EE)$ w.r.t. the $\EE_1$-norm, i.e. for
any $F\in\D(\EE)$, one may find a sequence $\{F_n\}_{n\geq1}\subset
C_0(\GG_X^\mu)$ such that $\EE_1(F_n-F,F_n-F)\rightarrow0$ as
$n\rightarrow\infty$.

\medskip

Since $\BB_b(\GG_X^\mu)\cap\D(\EE)$ is dense in $\D(\EE)$ (see e.g. \cite[Proposition I.4.17]{MR92}), we may assume that $F\in \BB_b(\GG_X^\mu)$.
Moreover, since $C_0(\GG_X^\mu)$ is dense in $L^2(\GG_X^\mu;\pi_{\lambda,\mu})$, we may find a sequence
$\{F_n\}_{n\geq1}\subset C_0(\GG_X^\mu)$ such that
$\sup_{n\in\mathbb{N}}\|F_n\|_{L^\infty(\pi_{\lambda,\mu})}\le \|F\|_{L^\infty(\pi_{\lambda,\mu})}$ and
$\pi_{\lambda,\mu}(|F_n-F|^2)\to 0$ as $n\to\infty.$  Without loss of
generality, we assume furthermore that $F_n\rightarrow F$
$\pi_{\lambda,\mu}$-a.e. By Lemma \ref{quasi}, $F_n(\gg+\dd_x)\to F(\gg+\dd_x)$ and
$(F_n-F)^2(\gg+\delta_x)\leq(\|F_n\|_{L^\infty(\pi_{\lambda,\mu})}+\|F\|_{L^\infty(\pi_{\lambda,\mu})})^2\leq4\|F\|_{L^\infty(\pi_{\lambda,\mu})}^2$ for
$\left(\int_{\R^+}\pi_{s\mu}s\lambda(\d s)\right)\times\mu$-a.e. $(\gg,x)\in\GG_X\times X$.

\medskip

Note that (we do not have to distinguish integrals on $\GG_X^\mu$ and $\GG_X$ since $\pi_{\lambda,\mu}(\GG\setminus\GG_X^\mu)=0$)
\begin{eqnarray*}
\begin{aligned}
&\EE(F_n-F,F_n-F)\\
&\leq2\int_{\GG_X\times X\times\R^+}(F_n-F)^2(\gg+\dd_x)\,\pi_{s\mu}(\d\gg)s\mu(\d x)\lambda(\d s)\\
&\quad+2\int_{\GG_X\times X\times\R^+}(F_n-F)^2(\gg)\,\pi_{s\mu}(\d\gg)s\mu(\d x)\lambda(\d s).
\end{aligned}
\end{eqnarray*}
By the dominated convergence theorem we obtain
$$\lim_{n\to\infty} \EE(F_n-F, F_n-F)=0.$$
Combining this with $\pi_{\lambda,\mu}(|F_n-F|^2)\rightarrow0$, we conclude that
$$\lim_{n\rightarrow\infty}\EE_1(F_n-F,F_n-F)=0,$$
which completes the proof.
\end{proof}

\section{Proofs of Theorems \ref{PI} and \ref{WPI} and Proposition \ref{PROP}}

To prove Theorem \ref{PI}, let's first make some preparations.
\begin{lem}\label{abso}
If $\mu(X)<\infty$, then for any $s,t\in\R^+$, $\pi_{s\mu}$ and $\pi_{t\mu}$ are mutually absolutely continuous.
More precisely,
$$
\frac{\d\pi_{t\mu}}{\d\pi_{s\mu}}(\gg)=\exp\left[\gg(X)\log\frac{t}{s}+(s-t)\mu(X)\right]
$$
for $\pi_{t\mu}$-a.e. (and so also $\pi_{s\mu}$-a.e.) $\gg\in\GG_X$.
\end{lem}
\begin{proof}
For every $f\in L^1(\mu)\cap L^\infty(\mu)$, it follows from (\ref{LT}) that
\begin{eqnarray*}
\begin{aligned}
\int_{\GG_X}\e^{\gg(f)}\,\pi_{t\mu}(\d\gg)&=\exp\left[t\mu(\e^f-1)\right]\\
&=\e^{(s-t)\mu(X)}\exp\left[s\mu(\e^{f+\log\frac{t}{s}}-1)\right]\\
&=\e^{(s-t)\mu(X)}\int_{\GG_X}\e^{\gg(f+\log\frac{t}{s})}\,\pi_{s\mu}(\d\gg)\\
&=\int_{\GG_X}\e^{\gg(f)}\exp\left[\gg(X)\log\frac{t}{s}+(s-t)\mu(X)\right]\,\pi_{s\mu}(\d\gg),
\end{aligned}
\end{eqnarray*}
from which we obtain the desired result.
\end{proof}

Let $\scr B_b(\GG_X)$ be the family of bounded measurable functions on $\GG_X$. The following result,
proved in \cite{HP08} (see also \cite{MZ96} for more general result),
is crucial for the proof of Theorems \ref{PI} and \ref{WPI}. Here we shall give another easy proof, which can be regarded
as an interesting application of Lemma \ref{abso}.
\begin{lem}\label{dera}
If $\mu(X)<\infty$, then for every $F\in\scr B_b(\GG_X)$ we have
$$
\frac{\d}{\d s}\pi_{s\mu}(F)=\int_{\GG_X\times X}\big(F(\gg+\delta_x)-F(\gg)\big)\,\pi_{s\mu}(\d\gg)\mu(\d x),\quad s\in\R^+.
$$
\end{lem}
\begin{proof}
By Lemma \ref{abso}, for any $s\in\R^+$ and $\varepsilon\geq-s$, one has
$$
\frac{\d\pi_{(s+\varepsilon)\mu}}{\d\pi_{s\mu}}(\gg)=\exp\left[\gg(X)\log\frac{s+\varepsilon}{s}-\varepsilon\mu(X)\right]
$$
for $\pi_{s\mu}$-a.e. $\gg\in\GG_X$. Combining this with the dominated convergence theorem, we arrive at
\begin{eqnarray*}
\begin{aligned}
\frac{\d}{\d s}\pi_{s\mu}(F)&=\lim_{\varepsilon\rightarrow0}\frac{1}{\varepsilon}
\left\{\int_{\GG_X} F(\gg)\,\pi_{(s+\varepsilon)\mu}(\d\gg)-\int_{\GG_X} F(\gg)\,\pi_{s\mu}(\d\gg)\right\}\\
&=\lim_{\varepsilon\rightarrow0}\frac{1}{\varepsilon}
\int_{\GG_X} F(\gg)\left\{\exp\left[\gg(X)\log\frac{s+\varepsilon}{s}-\varepsilon\mu(X)\right]-1\right\}\,\pi_{s\mu}(\d\gg)\\
&=\int_{\GG_X} F(\gg)\lim_{\varepsilon\rightarrow0}\frac{1}{\varepsilon}
\left\{\exp\left[\gg(X)\log\frac{s+\varepsilon}{s}-\varepsilon\mu(X)\right]-1\right\}\,\pi_{s\mu}(\d\gg)\\
&=\int_{\GG_X} F(\gg)\left(\frac{1}{s}\gg(X)-\mu(X)\right)\,\pi_{s\mu}(\d\gg)\\
&=\frac1s\int_{\GG_X\times X}F(\gg)\,\gg(\d x)\pi_{s\mu}(\d\gg)-\int_{\GG_X\times X}F(\gg)\,\pi_{s\mu}(\d\gg)\mu(\d x)\\
&=\int_{\GG_X\times X}\big(F(\gg+\delta_x)-F(\gg)\big)\,\pi_{s\mu}(\d\gg)\mu(\d x),
\end{aligned}
\end{eqnarray*}
where in the last step we have used the Mecke identity (\ref{mecke}) for $H(\gg,\cdot)=F(\gg)$.
\end{proof}

\begin{proof}[Proof of Theorem \ref{PI}]
By an approximation argument, we may and do assume that $F\in\scr B_b(\GG_X)$.
Recalling the Poincar\'{e} inequality for birth-death type Dirichlet form w.r.t. the Poisson measure (see \cite[Remark 1.4]{Wu00}), we have
$$
\pi_{s\mu}(F^2)\leq\pi_{s\mu}(F)^2+\int_{\GG_X\times X}(F(\gg+\delta_x)-F(\gg))^2\,\pi_{s\mu}(\d\gg)s\mu(\d x),\quad s\in\R^+.
$$
This yields that
\begin{eqnarray}\label{mm}
\begin{aligned}
\pi_{\lambda,\mu}(F^2)-\pi_{\lambda,\mu}(F)^2&=\int_{\R^+}\pi_{s\mu}(F^2)\,\lambda(\d s)
-\left(\int_{\R^+}\pi_{s\mu}(F)\,\lambda(\d s)\right)^2\\
&\leq\EE(F,F)+\int_{\R^+}\pi_{s\mu}(F)^2\,\lambda(\d s)-\left(\int_{\R^+}\pi_{s\mu}(F)\,\lambda(\d s)\right)^2.
\end{aligned}
\end{eqnarray}
Applying (\ref{con1}) with $f(s)=\pi_{s\mu}(F)$, and using Lemma \ref{dera} and the Cauchy-Schwartz inequality we obtain
\begin{eqnarray*}
\begin{aligned}
&\int_{\R^+}\pi_{s\mu}(F)^2\,\lambda(\d s)-\left(\int_{\R^+}\pi_{s\mu}(F)\,\lambda(\d s)\right)^2\\
&\leq C\int_{\R^+}s\left|\frac{\d}{\d s}\pi_{s\mu}(F)\right|^2\,\lambda(\d s)\\
&=C\int_{\R^+}s\left(\int_{\GG_X\times X}(F(\gg+\delta_x)-F(\gg))\,\pi_{s\mu}(\d\gg)\mu(\d x)\right)^2\,\lambda(\d s)\\
&\leq C\mu(X)\EE(F,F).
\end{aligned}
\end{eqnarray*}
Combining this with (\ref{mm}), the desired Poincar\'{e} inequality for $(\EE,\D(\EE))$ follows.
\end{proof}

\bigskip

\begin{proof}[Proof of Theorem \ref{WPI}]
Similarly as in the proof of Theorem \ref{PI}, we assume that $F\in\scr B_b(\GG_X)$.
Since $\|\pi_{\cdot\mu}(F)\|_{u}\leq\|F\|_{u}$, it holds from (\ref{con2}), Lemma \ref{dera} and the Cauchy-Schwartz inequality that
\begin{eqnarray*}
\begin{aligned}
&\int_{\R^+}\pi_{s\mu}(F)^2\,\lambda(\d s)-\left(\int_{\R^+}\pi_{s\mu}(F)\,\lambda(\d s)\right)^2\\
&\leq\alpha(r)\int_{\R^+}s\left|\frac{\d}{\d s}\pi_{s\mu}(F)\right|^2\,\lambda(\d s)+r\|F\|_{u}^2\\
&\leq\mu(X)\alpha(r)\EE(F,F)+r\|F\|_{u}^2,\quad r>0.
\end{aligned}
\end{eqnarray*}
Hence we finish the proof by combining this with (\ref{mm}).
\end{proof}

To prove Proposition \ref{PROP}, we need the following fundamental lemma. We
include a simple proof for completeness.

\begin{lem}\label{ss}
Assume that $\int_{\R^+}s^2\lambda(\d s)<\infty$. Then for $F(\gg)=\gg(f)$, where $f\in L^1(\mu)\cap L^\infty(\mu)$
and $f\geq0$ $\mu$-a.e., we have
$$
\pi_{\lambda,\mu}(F)=\mu(f)\int_{\R^+}s\,\lambda(\d s),\quad
\pi_{\lambda,\mu}(F^2)=\mu(f)^2\int_{\R^+}s^2\,\lambda(\d s)+\mu(f^2)\int_{\R^+}s\,\lambda(\d s).
$$
\end{lem}
\begin{proof}
Let $\varepsilon\leq0$. It follows from (\ref{LT}) that
$$
\int_{\GG_X}\e^{\varepsilon\gg(f)}\,\pi_{\lambda,\mu}(\d\gg)=\int_{\R^+}\exp\left[s\int_X\left(\e^{\varepsilon f}-1\right)\,\d\mu\right]\,\lambda(\d s).
$$
Since $\varepsilon\leq0$, $f\geq0$ and $\int_{\R^+}s^2\lambda(\d s)<\infty$, this implies that
$$
\int_{\GG_X}\e^{\varepsilon\gg(f)}\gg(f)\,\pi_{\lambda,\mu}(\d\gg)
=\int_{\R^+}s\exp\left[s\int_{X}\left(\e^{\varepsilon f}-1\right)\,\d\mu\right]\,\lambda(\d s)\times\int_{X}f\e^{\varepsilon f}\,\d\mu
$$
and
\begin{eqnarray*}
\begin{aligned}
\int_{\GG_X}\e^{\varepsilon\gg(f)}\gg(f)^2\,\pi_{\lambda,\mu}(\d\gg)&=\int_{\R^+}s^2\exp\left[s\int_{X}\left(\e^{\varepsilon f}-1\right)\,\d\mu\right]\,\lambda(\d s)\times\left(\int_{X}f\e^{\varepsilon f}\,\d\mu\right)^2\\
&\quad+\int_{\R^+}s\exp\left[s\int_{X}\left(\e^{\varepsilon f}-1\right)\,\d\mu\right]\,\lambda(\d s)\times\int_{X}f^2\e^{\varepsilon f}\,\d\mu.
\end{aligned}
\end{eqnarray*}
Now we finish the proof by taking $\varepsilon=0$ in these two equalities.
\end{proof}

\begin{proof}[Proof of Proposition \ref{PROP}]
Suppose that (\ref{PTI}) holds with $C_1<1$ and some $C_2$. Fix a nonnegative
$f\in L^1(\mu)\cap L^\infty(\mu)$ and let $F(\gg)=\gg(f)$. Then we deduce from (\ref{PTI}) and Lemma \ref{ss} that
$$
\mu(f)^2\int_{\R^+}s^2\,\lambda(\d s)+\mu(f^2)\int_{\R^+}s\,\lambda(\d s)
\leq C_1\mu(f^2)\int_{\R^+}s\,\lambda(\d s)+C_2\mu(f)^2\left(\int_{\R^+}s\,\lambda(\d s)\right)^2,
$$
which yields that
$$
\left[(1-C_1)\int_{\R^+}s\,\lambda(\d s)\right]\mu(f^2)
\leq\left[C_2\left(\int_{\R^+}s\,\lambda(\d s)\right)^2-\int_{\R^+}s^2\,\lambda(\d s)\right]\mu(f)^2
$$
holds for any nonnegative $f\in L^1(\mu)\cap L^\infty(\mu)$.
Noting that $(1-C_1)\int_{\R^+}s\,\lambda(\d s)\in(0,\infty)$, this is impossible according to the assumption that we can choose a sequence $\{A_n\}_{n\geq1}$ such that $\mu(A_n)>0$
but $\mu(A_n)\rightarrow0$ as $n\rightarrow\infty$.
\end{proof}

\end{document}